\documentclass[11pt,a4paper]{article}

\usepackage{amsmath,amsthm,amssymb,amscd,a4wide}

\usepackage{dsfont}
\usepackage{mathrsfs}
\usepackage{setspace}
\usepackage{hyperref}

\numberwithin{equation}{section}

\newtheorem{theorem}{Theorem}[section]
\newtheorem{lemma}[theorem]{Lemma}
\newtheorem{prop}[theorem]{Proposition}

\newtheorem{remark}[theorem]{Remark}

\theoremstyle{definition}
\newtheorem{defn}[theorem]{Definition}

\numberwithin{equation}{section}

\begin{document}
\thispagestyle{empty}

\vspace*{1cm}

\begin{center}
	
	{\LARGE\bf Numerical range with respect to a family of projections} \\

	\vspace*{2cm}
	
	{\large Waed Dada \footnote{E-mail address: {\tt waed.dada@fernuni-hagen.de}} and Joachim Kerner \footnote{E-mail address: {\tt joachim.kerner@fernuni-hagen.de}} }
	
	\vspace*{5mm}
	
	Department of Mathematics and Computer Science\\
	FernUniversit\"{a}t in Hagen\\
	58084 Hagen\\
	Germany\\
	
		\vspace*{5mm}
	
{\large and }

	\vspace*{5mm}
	
{\large N. Erkur\c{s}un \footnote{E-mail address: {\tt naer@fa.uni-tuebingen.de}}}

\vspace*{5mm}

Department of Mathematics\\
Hacettepe University\\
Ankara\\
Turkey\\
\end{center}

\vfill

\begin{abstract}
	In this note we introduce the concept of the numerical range of a bounded linear operator with respect to a family of projections. We give a precise definition and elaborate on its connection to the classical numerical range as well as to generalisations such as the quadratic numerical range and block numerical range.
\vspace{0.2cm}

\end{abstract}

\section{Introduction} 
The numerical range $W(A)$ of a bounded linear operator $A$ on a complex Hilbert space $H$ is 
\begin{equation}
W(A):=\{\langle Ax,x\rangle: \ x \in H\ ,\  \|x\|=1 \} \subset \mathbb{C}\ .
\end{equation}
Originally introduced for linear operators on $\mathbb{C}^n$ by Toeplitz \cite{to18} and Hausdorff \cite{ha19}, it was later extended to operators on Hilbert spaces by Stone \cite{st32}. Unlike the spectrum, the numerical range is a unitary invariant but in general not invariant under similarity transformations and hence provides additional information about the operator. Furthermore, since the numerical range is relatively easy to compute (at least in the matrix 
case), it became a useful tool in many applications \cite{gr97,sh03,tre03}. 

The following list contains the basic properties of the numerical range for operators $A \in \mathcal{L}(H)$:
\begin{enumerate}
\item $W(A) \subseteq \{ \lambda \in \mathbb C: |\lambda| \leq \| A\| \}$.
\item   $ W(U^\ast A U)=W(A) $ for  any unitary $U \in \mathcal{L}(H)$.
\item For a two-dimensional Hilbert space $H$, $W(A) \subset \mathbb{C}$ is a (possibly degenerate) ellipse.
\item If $H$ is finite-dimensional, then $W(A) \subset \mathbb{C}$ is compact.
\item $\sigma(A)\subseteq \overline{W(A)}$ and  $\sigma_p(A)\subseteq W(A)$. \quad (Spectral Inclusion)
\item $W(A) \subset \mathbb{C}$ is convex. \quad (Toeplitz-Hausdorff Theorem)
\end{enumerate} 
Furthermore, as shown in \cite{bg72}, if $A \in \mathcal{L}(H)$ is a compact operator, then $W(A)$ is closed if $0 \in W(A)$ and if, in addition, $A$ is self-adjoint, then $W(A)$ is the convex hull of the point spectrum of $A$.
\section{Preliminaries and definitions}
We define the numerical range of a linear bounded operator $A$ with respect to families of orthogonal projections and consider, for $k \in \mathbb{N}$,
\begin{equation}\begin{split}
\mathbb{P} &:=\{ P \in \mathcal L (H) : P\mbox{ orthogonal projection in}\; H \}\ ,
\end{split}
\end{equation}
\begin{equation}\begin{split}
\mathcal{P}_k &:=\{ P \in \mathbb{P}:\dim(\mathrm{ran}(P))=k \}\ ,
\end{split}
\end{equation}
as well as,
\begin{equation}\begin{split}
\mathcal P_A:=\{ P\in \mathbb{P}: PA=AP\ , \ \dim (\mathrm{ran}(P)) < \infty\} .
\end{split}
\end{equation}
\begin{prop} $\mathbb{P},\mathcal{P}_k \subset \mathcal{L}(H)$ are closed sets with respect to the operator norm.
\end{prop}
%
%
\begin{defn} For $A\in \mathcal L(H)$ and $P \in  \mathbb P$ we define an operator $A_P $ on $\mathrm{ran}(P)$ by 
$$A_P: \mathrm{ran}(P) \to \mathrm{ran}(P)\ ,\  \; x\mapsto A_Px:= PAx\ .$$
The relation between $A_P$ and $A$ can be expressed by $$ A_PP=PAP.$$
The operator $A_P$ is called the compression of $A$ to $\mathrm{ran}(P)$ and $A$ is called a dilation of $A_P$ to $H$.
\end{defn}

\begin{remark}\label{Remark1}
	We have  $ W(A_P)\subset W(A)$: For any $\lambda\in W(A_P)$ there exists $x\in \mathrm{ran}(P)$ with $\|x\|=1$, $Px=x$ and  
	$$\lambda= \langle A_P x,x\rangle= \langle PAPx,x\rangle= \langle APx,Px\rangle =\langle Ay,y\rangle,$$
	where $y:=Px$. Since $\|y\|=\|Px\|=\|x\|=1$ we conclude $\lambda \in W(A)$. 
\end{remark}
Due to the spectral inclusion (see $5.$ in the list above) Remark~\ref{Remark1} implies, in particular, that each $\lambda \in \sigma(A_P)$ is contained in $W(A)$ given that $\dim(\mathrm{ran}(P)) < \infty$. This motivates the following definition.  
\begin{defn}[Numerical range with respect to a family of projections]
Let $A\in \mathcal L(H)$ be a bounded operator and $\mathcal P\subseteq \mathbb P$. Then 
\begin{equation}\label{PNumericalRange}
W_{\mathcal P}(A) := \bigcup_{P\in \mathcal P} \sigma (A_P)
\end{equation}
is called the numerical range of $A$ with respect to the family of orthogonal projections $\mathcal P$ or $\mathcal P$-numerical range of $A$ for short.
\end{defn} 
\begin{remark}
	\begin{enumerate} 
		\item[]
		\item $W_{\mathcal P}(A)$ is in general not closed\ .
		\item $W_{\mathcal P}(A)\subset \{\lambda\in \mathbb{C}: |\lambda|\leq ||A|| \}$\ .
\item $W_{\mathcal P}(A^\ast)= \big (W_{\mathcal P}(A)\big)^\ast$\ .
	
	\end{enumerate}
\end{remark}

\section{Main Results}
\subsection{Connection to the classical numerical range and the point spectrum}
The first result establishes the connection with the classical numerical range. For the proof note that, for each $P\in \mathcal P_k$ and for any orthonormal basis $\{f_i\}_{i=1}^k$ of $\mathrm{ran}(P)$, one has 
\begin{equation}\label{EquationProof}
P x= \sum_{i=1}^k\langle x,f_i\rangle f_i\  \quad  \forall x\in H.
\end{equation}
\begin{prop}\label{one_dim} For $A\in \mathcal L(H)$ we have  $W_{\mathcal P_1}(A)= W(A)$.
\end{prop}
\begin{proof}
Let $\lambda \in W_{\mathcal P_1}(A)$. Then there exist $P\in \mathcal P_1$ and $f \in \mathrm{ran}(P)$ with $\|f\|=1$ such that $PAP f=\lambda f$. Therefore
\begin{equation}
\langle Af,f\rangle= \langle APf,Pf\rangle=\langle PAPf,f\rangle =\langle \lambda f,f\rangle=\lambda
\end{equation}
and hence $\lambda \in W(A)$.\\
If $\lambda \in W(A)$ then there exists $f\in H$ with $||f||=1$ such that $\lambda =\langle Af,f\rangle$. Let $P$ denote the orthogonal projection onto span\{f\}. Then, according to \eqref{EquationProof}, 
\begin{equation}
PAP f= PAf= \langle Af,f\rangle f = \lambda f
\end{equation}
and hence $\lambda \in \sigma(A_P)$. Thus $\lambda \in W_{\mathcal P_1}(A)$.
\end{proof}
\begin{remark}Lemma~\ref{one_dim} is interesting from the following point of view: In general the spectrum forms only a ``small'' subset of $W(A)$ (for example, think of a matrix $A \in \mathbb{C}^2$ for which the spectrum consists of two points whereas $W(A)$ is a (possibly degenerate) ellipse). However, by considering the union of all $\sigma(A_P)$ for $P \in \mathcal{P}_1$ instead, the whole classical numerical range is obtained by ``filling it up'' with spectral values.
\end{remark}
The following statement is a direct generalisation of Proposition~\ref{one_dim}.
\begin{lemma}\label{11}
For $A\in \mathcal L(H)$ and the family $\mathcal P_k$ with $k \in \mathbb{N}$ the following holds:
\begin{enumerate}
\item If dim $H= k$, then $W_{\mathcal P_k}(A)= \sigma(A).$
\item If dim $H> k$, then  
$W_{\mathcal P_k}(A)= W(A)$.
\end{enumerate}
\end{lemma}
\begin{proof} 
$1.$ is a direct consequence of $\mathcal P_k=\{\mathfrak Id\}$ given dim $H= k$.\\
 Regarding $2.$, let $\lambda \in W(A)$ be given. Then there exists $f_0\in H$ with $||f_0||=1$ such that $\lambda =\langle Af_0,f_0\rangle$. Now take $f_1,f_2,\cdots,f_{k-1} \in H$ with $||f_i||=1$ such that $f_i \perp f_j$ , $i\neq j$, for $i,j=0,1,\cdots,k-1$ and $f_i\perp Af_0$ and $i=1,\cdots,k-1$.
Let $P$ be the orthogonal projection onto span$\{f_0,f_1,\cdots, f_{k-1} \}$ which is a $k$-dimensional subspace. Then, employing \eqref{EquationProof},
\begin{equation}
\begin{split}
PAP f_0=PAf_0 &= \langle Af_0,f_0\rangle f_0+ \langle Af_0,f_1\rangle f_1+\cdots +\langle Af_0,f_{k-1}\rangle f_{k-1} \\
&= \langle Af_0,f_0 \rangle f_0 \\
&= \lambda f_0\ ,
\end{split}
\end{equation}
i.e., $\lambda$ is an eigenvalue of $PAP$. Hence $\lambda \in W_{\mathcal P_k}(A)$.\\
Now take $\lambda \in W_{\mathcal P_k}(A)$. Then there exist $P\in \mathcal P_k$ and $f\in \mathrm{ran}(P)$ with  $\|f\|=1$ such that $PAP f= \lambda f$. Hence $\langle Af,f\rangle =\langle PAP f,f\rangle = \langle \lambda f,f\rangle = \lambda$, implying $\lambda \in W(A)$.
\end{proof}
%
%
In the next result we show how the family of projections $\mathcal{P}_A$ is related to the point spectrum of the operator $A$.
\begin{theorem}
For arbitrary $A\in \mathcal L(H)$ one has $W_{\mathcal P_{A}}(A) \subset \sigma _P(A)$. Furthermore, if $A$ is symmetric then $W_{\mathcal P_{A}}(A)=\sigma _P(A)$.
\end{theorem}
\begin{proof}
	For arbitrary $A$, let $\lambda \in W_{\mathcal P_{A}}(A)$ be given. Then there exists $P\in \mathcal P_{A}$ and $0\ne f\in \mathrm{ran}(P)$ such that $PAPf=\lambda f$ and
	$PAPf= APPf=APf$, since $f\in \mathrm{ran}(P)$. We obtain $APf=Af=\lambda f$ 
	and hence $\lambda \in \sigma_P(A)$.
	
Now let $A$ be symmetric: Then, for $\lambda \in \sigma_P(A)$ there exists a normalised $f\in H$ such that $Af=\lambda f$. Now, choose $P$ to be the orthogonal projection onto $\text{span}\{f\}$. Applying $P$ to the eigenvalue equation directly yields $PAPf=\lambda Pf=\lambda f$ which shows that $\lambda$ is an eigenvalue to $PAP$. On the other hand, since $A$ is symmetric one has, $x \in H$,
\begin{equation}
PAx=\langle Ax,f\rangle f=\langle x,f\rangle \lambda f=APx\ .
\end{equation}
This implies that $P \in \mathcal{P}_A$ and consequently $\lambda \in W_{\mathcal P_{A}}(A)$.\\

\end{proof}
\subsection{Connection to the quadratic and block numerical range}
As defined in \cite{LangerSpectralDecomposition} (and discussed in detail in \cite{lmm01}), the quadratic numerical range of a $2 \times 2$-block operator matrix 
\begin{equation}\label{BlockOperator}
\mathcal{A}=\begin{pmatrix}
A & B \\
C & D
\end{pmatrix}\ ,
\end{equation}
with $\mathcal{A}$ acting as an operator on $H_1 \oplus H_2$, is the set of all eigenvalues of all $2\times 2$-matrices 
\begin{equation}
\mathcal{A}_{f,g}=\begin{pmatrix}
\langle Af,f\rangle & \langle Bg,f\rangle \\
\langle Cf,g\rangle & \langle Dg,g\rangle
\end{pmatrix}
\end{equation}
with $f \in H_1$, $g \in H_2$ and $\|f\|=\|g\|=1$. The quadratic numerical range of $\mathcal{A}$ will be denoted by $W_{H_1,H_2}(\mathcal{A})$. 

In order to relate the quadratic numerical range to a family of projections, one considers the set of all projections $P \in \mathcal{P}_2$ such that $\mathrm{ran}(P)$ has dimension two and is spanned by two vectors in $H_1 \oplus H_2$ of the form $F_1:=f_1 \oplus 0$, $F_2:=0 \oplus f_2$ with (non-zero) $f_1 \in H_1$ and $f_2 \in H_2$. We will denote this family of projections by $\mathcal{P}_{H_1,H_2}$.

For any such $P \in \mathcal{P}_{H_1,H_2}$ we obtain
\begin{equation}\begin{split}
\mathcal{A}_PF_1:=P\mathcal{A}F_1=\langle \mathcal{A}F_1,F_1\rangle F_1 +\langle \mathcal{A}F_1,F_2\rangle F_2\ ,
\end{split}
\end{equation}
and
\begin{equation}\begin{split}
\mathcal{A}_PF_2:=P\mathcal{A}F_2=\langle \mathcal{A}F_2,F_1\rangle F_1 +\langle \mathcal{A}F_1,F_2\rangle F_2\ .
\end{split}
\end{equation}
Accordingly, $\mathcal{A}_P$ can be represented by a $2\times2$ matrix  with respect to this basis as
\begin{equation}
\mathcal{A}_P:=\begin{pmatrix} \langle \mathcal{A}F_1,F_1\rangle &\langle \mathcal{A}F_1,F_2\rangle\\ \langle \mathcal{A}F_2,F_1\rangle &\langle \mathcal{A}F_2,F_2\rangle \end{pmatrix}\in M_{2\times 2}(\mathbb C).
\end{equation}
Furthermore, a direct calculation shows that 
\begin{equation}\label{EquationProofXXX}
\mathcal{A}_P=\begin{pmatrix} \langle Af_1,f_1\rangle &\langle Cf_1,f_2\rangle\\ \langle Bf_2,f_1\rangle &\langle Df_2,f_2\rangle \end{pmatrix}=\mathcal{A}^T_{f_1,f_2}.
\end{equation}
This allows us to establish the following result.
\begin{theorem}\label{Theorem2NumericalRange}
Let $H=H_1\oplus H_2$ be a Hilbert space and $\mathcal{A}\in \mathcal L(H)$ a block operator of the form \eqref{BlockOperator}. Then 
\begin{equation}
W_{\mathcal P_{H_1,H_2}}(\mathcal{A})= W_{H_1,H_2}(\mathcal{A})\ .
\end{equation}
\end{theorem}
\begin{proof}
For $\lambda \in W_{H_1,H_2}(\mathcal{A})$ and by definition of the quadratic numerical range there exist (normalised) $f_1\in H_1$, $f_2\in H_2$ and a (non-zero) $h\in\mathbb C^2$ such that $\mathcal{A}_{f_1,f_2} h=\lambda h$. Let $P$ denote the orthogonal projection onto $\mathrm{span}\{f_1\oplus 0, 0 \oplus f_2\}$. Then $P\in \mathcal P_{H_1,H_2}$ and, according to \eqref{EquationProofXXX}, $\sigma(\mathcal{A}_P)=\sigma(\mathcal{A}^T_{f_1,f_2})$ which implies $\lambda \in W_{\mathcal P_{H_1,H_2}}(\mathcal{A})$.

Now, let $\lambda \in W_{\mathcal P_{H_1,H_2}}(\mathcal{A})$ be given. Then there exists a projection $P \in  \mathcal P_{H_1,H_2}$ and a (non-zero) element $h \in \mathrm{ran}(P)$ such that $\mathcal{A}_Ph=\lambda h$. Employing relation \eqref{EquationProofXXX} again yields $\lambda \in W_{H_1,H_2}(\mathcal{A})$. Note that the existence of corresponding (normalised) vectors $f_1 \in H_1$, $f_2 \in H_2$ follows from the definition of the family $\mathcal{P}_{H_1,H_2}$.
\end{proof}
Theorem~\ref{Theorem2NumericalRange} can be directly generalised to $k$-block operators $\mathcal{A}$ acting on a Hilbert space of the form $H=\bigoplus_{i=1}^k H_i$ by defining the family $\mathcal{P}_{H_1,...,H_k}$ of projections in analogy to the case of $k=2$. Furthermore, in analogy to the quadratic numerical range one introduces the \textit{block numerical range} $W_{H_1,...,H_k}(\mathcal{A})$ (see also \cite{tre03,tre08}) and can obtain the following result.
\begin{theorem}\label{TheoremBlockNumericalRange}
	Let $H=\bigoplus_{i=1}^k H_i$ be a Hilbert space and $\mathcal{A}\in \mathcal L(H)$ a block operator on $H$, i.e., $(\mathcal{A})_{1\leq i,j\leq k}=A_{ij}$ with $A_{ij}:H_j \rightarrow H_i$ bounded linear operators. 
	Then 
	\begin{equation}
	W_{\mathcal P_{H_1,...,H_k}}(\mathcal{A})= W_{H_1,...,H_k}(\mathcal{A})\ .
	\end{equation}
\end{theorem}
\begin{remark} Regarding Theorem~\ref{Theorem2NumericalRange} and Theorem~\ref{TheoremBlockNumericalRange} we observe the following: If $A$ is a $n\times n$-matrix acting on $\mathbb{C}^n$, we can divide it into blocks as to obtain a $k$-block operator acting on $\mathbb{C}^{n_1}\oplus ... \oplus \mathbb{C}^{n_k}$ with $n_1+...+n_k=n$. Also, dividing each $\mathbb{C}^{n_j}$ further and hence obtaining a refined partition of $\mathbb{C}^n$ yields a $p$-block operator acting on $\mathbb{C}^{\tilde{n}_1}\oplus ... \oplus \mathbb{C}^{\tilde{n}_p}$ with $\tilde{n}_1+...+\tilde{n}_p=n$, $p > k$. Denoting all three operators by $A$, the definition of the families $\mathcal{P}_{H_1,...,H_k}$ from above allows us to obtain the inclusion 
	\begin{equation}
	W_{\mathcal P_{\mathbb{C}^{\tilde{n}_1},...,\mathbb{C}^{\tilde{n}_p}}}(A) \subseteq W_{\mathcal P_{\mathbb{C}^{n_1},...,\mathbb{C}^{n_k}}}(A)\subseteq W_{\mathcal P_{1}}(A)=W(A)\ ,
	\end{equation}
where the last equality is due to Proposition~\ref{one_dim}. Regarding the block numerical ranges we therefore obtain the inclusion
	\begin{equation}\label{EquationTretterInclusion}
	W_{\mathbb{C}^{\tilde{n}_1},...,\mathbb{C}^{\tilde{n}_p}}(A) \subseteq W_{{\mathbb{C}^{n_1},...,\mathbb{C}^{n_k}}}(A)\subseteq W(A)\ .
	\end{equation}
It is interesting to note that equation~\eqref{EquationTretterInclusion} was already obtained in \cite{tre03} by different methods.
\end{remark}
\vspace{1cm}


\newpage
\small{

}

\end{document}